\documentclass[12pt]{article}

\usepackage[a4paper, margin=1in]{geometry}  
\usepackage{graphicx}              	
\usepackage{amsmath,accents}               	
\usepackage{amsfonts}              	
\usepackage{amsthm}                	
\usepackage{amssymb, eurosym}      	
\usepackage[utf8x]{inputenc}	   	
\usepackage[main=english,brazil]{babel}         	
\usepackage{indentfirst}		   	
\usepackage{bbm}					
\usepackage{blindtext}				
\usepackage{titling}				
\usepackage{mathrsfs}				
\usepackage{enumitem}				
\usepackage{hyperref}
\usepackage{xcolor}

\newcommand{\subtitle}[1]{%
	\posttitle{%
		\par\end{center}
	\begin{center}\large#1\end{center}
	\vskip0.5em}%
}

\newtheorem{thm}{Theorem}[section]
\newtheorem{lemma}[thm]{Lemma}
\newtheorem{prop}[thm]{Proposition}
\newtheorem{cor}[thm]{Corollary}
\newtheorem{definition}[thm]{Definition}

\newtheorem{per}[thm]{Question}
\newtheorem{remark}[thm]{Remark}

\newcommand{\forces}{\mathrel{\Vdash}} 
\newcommand{\1}{\mathbbm{1}}      

\DeclareMathOperator{\diff}{diff}
\DeclareMathOperator{\acc}{acc}

\DeclareMathOperator{\domOp}{dom}
\DeclareMathOperator{\ranOp}{ran}

\newcommand\dom[1]{\domOp\left({#1}\right)}
\newcommand\ran[1]{\ranOp\left({#1}\right)}

\title{Partitions of topological spaces and a new club-like principle}
\author{Rodrigo Carvalho \\ Gabriel Fernandes\\ Lúcia R. Junqueira}
\date{April 2022}

\begin{document}

\maketitle
    
\begin{abstract}	We give a new proof of the following theorem due to W. Weiss and P. Komjath: if $X$ is a regular topological space, with character  $ < \mathfrak{b}$ and $X \rightarrow (top \omega + 1)^{1}_{\omega}$, then, for all $\alpha < \omega_1$, $X \rightarrow (top \alpha)^{1}_{\omega}$, fixing a gap in the original one. For that we consider a new decomposition of topological spaces. We also define a new combinatorial principle $\clubsuit_{F}$, and use it to prove that it is consistent with $\neg CH$ that $\mathfrak{b}$ is the optimal bound for the character of $X$. In \cite{WeissKomjath}, this was obtained using $\diamondsuit$.
\end{abstract}






\section{Introduction}

The topic of partitions of topological spaces is a nice example of the interaction of Ramsey theory and topology. One could consider the early appearances of this topic to date back to the 70's in the works of \cite{friedman1974closed} and \cite{prikry1975partitions} as well as some problems of \cite{102probmathlog}, dealing mainly with ordinals. Later this study would be expanded to properly include topological spaces as in \cite{nevsetvril1977ramsey} and a later discussion in \cite{PartThAppl}.  For a survey on  this topic we refer to  \cite{PartWeiss}.\\

\begin{definition}
	Let $X$ be a topological space and $\alpha$ an ordinal. The expression $$X \rightarrow (top \, \alpha)^{1}_{\kappa}$$ means that, for all $f \in \kappa^{X}$, there is an $\eta \in \kappa$ and a subset $Y \subset f^{-1}[\{\eta\}]$ such that $Y$ is homeomorphic to $\alpha$ with the order topology.
\end{definition}

In the first section of this paper we give a new proof for the following thm, due to W. Weiss and P. Komjáth, fixing an oversight in the original. We would like to thank W. Weiss for the suggestion given on how to fix the original proof. 

\begin{thm}\cite[Theorem 1]{WeissKomjath}\label{WKoriginal}
	Let $X$ be a regular topological space with $X \rightarrow (top \, \omega+1)^{1}_{\omega}$, and $\chi(X) < \mathfrak{b}$. Then $X \rightarrow (top \, \alpha)^{1}_{\omega}$ for all $\alpha < \omega_{1}$.
\end{thm}

In the proof we revisit the concept of the Cantor-Bendixson decomposition as a process that removes undesirable elements in our topological space. This will allow us to question what undesirability means (an isolated point in the classical definition), and consider a new construction considering convergent sequences.\\

In  \cite[Theorem 4]{WeissKomjath},  using $\diamondsuit$, they construct an example that we can not improve Theorem \ref{WKoriginal} allowing on its hypothesis that $\chi(X)=\mathfrak{b}$.  

\begin{thm}\cite[Theorem 4]{WeissKomjath} \label{RamseyOriginal} Suppose $\diamondsuit$ holds. Then there exists a topological space $(X,\tau)$ such that:  $\chi(X)=\mathfrak{b}$, $(X,\tau)\rightarrow (\omega+1)^{1}_{\omega}$,
	and $(X,\tau)\nrightarrow(\omega^{2}+1)^{1}_{\omega}.$\\
\end{thm} 

It is natural then to ask whether we can obtain a space $(X,\tau)$ with the Ramsey properties mentioned above without the full power of $\diamondsuit$ or even without CH. To the best of our knowledge it was still unknown whether such space could be constructed without it. In a private communication W. Weiss asked the third author if it was possible to get an example just using $\clubsuit$ \\

To answer that, in section \ref{Comb} we define a new combinatorial principle $\clubsuit_{F}$. Such principle captures the main usage of $\diamondsuit$ in the proof of Theorem \ref{RamseyOriginal}. In section \ref{Comb} we prove that $\clubsuit_{F}$ is a consequence of $\diamondsuit^{*}$ and the consistency of $\clubsuit_{F}$ with $\neg CH$. For this we reference the $CS^{*}$-iterated forcings done in \cite{fuchino} and \cite{chen}.\\

In section \ref{Top} we use this new principle to get what we wanted Theorem \ref{Cor1} and Theorem \ref{main1} show that it is consistent to have a topological space $(X,\tau)$ with the above Ramsey properties together with $\chi(X)=\mathfrak{b} = \omega_{1}$ and $\neg CH$.  \\

\section{Sequential Cantor-Bendixson and Partitions}\label{CB}

In this section, we construct what we call the sequential Cantor-Bendixson decomposition, following the suggestion given by W. Weiss. Instead of considering only isolated points to be undesirable, we also do so for the points in our topological space that do not have $\omega$-sequences converging to them. The next definitions are in parallel with the original notation. 

\begin{definition}
	
	Given a topological space $Y$, using recursion on $\alpha$ ordinal define:\\

	$SI_{0}(Y) = \{x \in Y :$ there is no $s:\omega \rightarrow Y$ injective sequence converging to $x\}$, and \\
	
	$SI_{\alpha}(Y) = SI_{0}(Y^{\alpha}_{S})$, where $Y^{\alpha}_{S} = Y \setminus \bigcup\{SI_{\beta}(Y) : \beta < \alpha\}$ \\
	
	We set the sequential height of a topological space $Y$ as\\
	
	$Sh(Y) = min\{\alpha : SI_{\alpha}(Y) = \emptyset\}$.
	
\end{definition}

Note that, as in the classical construction, the height is well defined. Indeed, $\{\alpha : SI_{\alpha}(Y) = \emptyset \} \neq \emptyset$ since $|Y|^{+}$ belongs to it.

\begin{definition}
	
	A topological space $Y$ is said to be S-scattered if $Y^{Sh(Y)}_{S} = \emptyset$ or, equivalently, if $Y = \bigcup\{SI_{\alpha}(Y) : \alpha < Sh(Y)\}$.
	
\end{definition}

Observe that if $\alpha > 0$, then, for all $x \in SI_{\alpha}(Y)$, there must be an injective sequence $s:\omega \rightarrow Y$, such that $s$ converges to $x$ and $s[\omega] \subset \bigcup\{SI_{\beta}(Y) : \beta < \alpha\}$. As we show next:

\begin{prop}\label{PropViz}
	
	Let $Y$ be a S-scattered topological space, $\alpha > 0$ and $x \in SI_{\alpha}(Y)$. Then, for all $\beta < \alpha$, there is an injective sequence $s: \omega \rightarrow Y$ converging to $x$, such that\linebreak $s[\omega] \subset \bigcup\{SI_{\gamma}(Y) : \gamma \in [\beta, \alpha)\}$.
	
\end{prop}

\begin{proof}
	
	Suppose that there is a $\beta < \alpha$ such that, for all injective sequences $s$ converging to $x$, we have $s[\omega] \not\subset \bigcup\{SI_{\gamma}(Y) : \gamma \in [\beta, \alpha)\}$. In particular, we must have\linebreak $|s[\omega] \cap (\bigcup\{SI_{\gamma}(Y) : \gamma \in [\beta, Sh(Y))\})| < \aleph_{0}$; otherwise, we could take a convergent sub-sequence that contradicts either the hypothesis or the fact that $x \in SI_{\alpha}(Y)$.
	
	This condition assures us that there are no sequences in $Y^{\beta}_{S}$ converging to $x$ and, therefore, in the worst scenario, $x \in SI_{\beta}(Y)$. That is an absurd; hence, the thesis must hold.
	
\end{proof}

Analogous to the scattered and perfect decomposition of a space, we define:
\begin{definition}
	A topological space $Y$ is S-perfect if, for all $x \in Y$, there is an injective sequence converging to $x$.
\end{definition} 

Note that the sequential Cantor-Bendixson decomposition of a space gives the following:

\begin{cor}
	Given a topological space $Y$, $Y^{Sh(Y)}_{S}$ is S-perfect.
\end{cor}

We also have the following property:

\begin{prop}
	Let $Y$ be a topological space and $\alpha < Sh(Y)$. If $\alpha + 1 < Sh(Y)$ then $|SI_{\alpha}(Y)| \geq \aleph_{0}$. 
\end{prop}

\begin{proof}
	Note that $\alpha + 1 < Sh(Y)$ implies $SI_{\alpha+1}(Y) \neq \emptyset$. Let $x \in SI_{\alpha+1}$. By Proposition \ref{PropViz}, we have $s:\omega \rightarrow Y$ an injective sequence such that $s$ converges to $x$ and $s[\omega] \subset SI_{\alpha}(Y)$. Now $\aleph_{0} \leq |s[\omega]| \leq |SI_{\alpha}(Y)|$.
	
\end{proof}

Next we shall address the problem found in the proof of \cite[Theorem 1]{WeissKomjath}. Using the regular Cantor-Bendixson decomposition and the restriction on the character it is not possible to guarantee the existence of a converging sequence that was needed in certain steps of the construction. Hence we need to work directly with sequences in the decomposition.

We will first prove an auxiliary lemma, using a similar argument done in \cite{WeissKomjath}.

\begin{lemma} \label{Alpha}
	Let $X$ be a regular, S-perfect, non empty topological space. If $\chi(X) < \mathfrak{b}$, then, for all $\alpha < \omega_{1}$, and for all open set $V \subseteq X$ there exists $\Phi:\alpha \rightarrow V$ such that $\Phi$ is a homeomorphism. 
\end{lemma}

\begin{proof}
	The copies of $\alpha$ will be recursively constructed. Fix $\alpha$ and suppose that, for all $\beta < \alpha$, $x \in X$ and open neighbourhoods $V$ of $x$, there are $f_{\beta}: \beta \rightarrow V$ homeomorphism such that if $\beta$ is the successor of a limit ordinal $\gamma$ then $f_{\beta}(\gamma) = x$. Fix $x \in X$ and $V$ open neighbourhood of $x$.
	
	Using the cantor normal decomposition consider $\alpha = \omega^{\beta_{1}}.n_{1} + \cdots + \omega^{\beta_{k}}.n_{k}$. We will analyse two cases, $\alpha = \omega^{\beta_{1}}.1 + 1$ or $\alpha = \omega^{\beta_{1}}.1$, and the case otherwise. For the latter case, using the S-perfectness fix an injective sequence $s$ converging to $x$ such that $s[\omega] \subset V$. By the injectiveness of $s$ and regularity of $X$, fix $y \in V$ and open sets $W_{1}$, $W_{2}$ satisfying $y \in W_{1} \subset \overline{W_{1}} \subset V$, $x \in W_{2} \subset \overline{W_{2}} \subset V$ with $\overline{W_{1}} \cap \overline{W_{2}} = \emptyset$. We can easily split $\alpha$ into two smaller ordinals and use the hypothesis of the recursion to obtain their copies in each of the disjoint open sets above. Using these copies we construct a copy of $\alpha$ as desired. 
	
	\vspace*{0,25cm}
	
	Now we address the first case by considering whether or not $\beta_{1}$ is a limit ordinal. Suppose $\beta_{1} = \gamma + 1$. Now $(\omega^{\gamma}.n : n \in \omega)$ is a sequence converging to $\omega^{\beta_{1}}$. By the S-perfectness of $X$, fix $s:\omega \rightarrow X$, an injective sequence in $V\setminus \{x\}$ converging to $x$. The regularity of $X$ yields open sets $\{U_{n} : n \in \omega\}$, such that $\overline{U_{i}} \cap \overline{U_{j}} = \emptyset$, if $i \neq j$, and $s(n) \in U_{n} \subset \overline{U_{n}} \subset V$ for all $n \in \omega$. Using the assumption made in the recursion, we take $f_{i}: \omega^{\gamma} + 1 \rightarrow f_{i}[\omega^{\gamma} + 1] \subset U_{i}$ homeomorphism such that $f_{i}(\omega^{\gamma}) = s(i)$.
	
	Fix a cofinal function $h: \omega \rightarrow \omega^{\gamma}$ and an open basis $\{W_{\rho} : \rho < \lambda < \mathfrak{b}\}$ for $x$. Let $p_{\rho} : \omega \rightarrow \omega$ be given by: 
	
	\[p_{\rho}(t) = \left\{ \begin{array}{ll}
		min\{i : \forall \delta \geq h(i) \, (f_{t}(\delta) \in W_{\rho})\},  & \mbox{if $s(t) \in W_{\rho}$}\\
		0,  & \mbox{otherwise.}
	\end{array} \right. \]
	
	Using the definition of $\mathfrak{b}$, we take $g:\omega\rightarrow\omega$ satisfying $\forall \rho < \lambda \, (p_{\rho} \leq^{*} g)$.
	
	\vspace*{0,25cm}
	
	The value $p_{\rho}(t)$ helps us understand how the open set $W_{\rho}$ behaves in relation with the homeomorphism $f_{t}$ using $h$. The function $g$ will help us to translate this behaviour to all the $W_{\rho}$'s.

	With this in mind, the desired copy of $\omega^{\beta_{1}} + 1$ (or $\omega^{\beta_{1}}$) in $V$ will be given by:
	
	\[f(\theta) = \left\{ \begin{array}{ll}
		f_{t}(h(g(t) + \theta')   & \mbox{if $\theta = \omega^{\gamma}.t + \theta'$ in the normal cantor form and $0 < \theta' \leq \omega^{\gamma}$} \\
		x  &  \mbox{if $\theta = \omega^{\gamma + 1}$}
	\end{array} \right. \]

	Now that $f$ is well defined and injective, we need only to prove that it is open and continuous. For the continuity, fix an open set $U$ in $X$. Without loss of generality we may assume $U \subset V$. Let $\lambda \in f^{-1}[U]$. If $\lambda \neq \omega^{\beta_{1}}$, we have that $\lambda = \omega^{\gamma}.t + \theta'$. Since $f(\lambda) \in U$ if and only if $f_{t}(h(g(t)) + \theta') \in U$ and  $f_{t}$ is continuous, we have the continuity of $f$ at $\lambda$. 
	If $\lambda = \omega^{\beta_{1}}$, then $f(\lambda) = x$. Consider $\rho$, such that $x \in W_{\rho} \subset U$. There must be $k \in \omega$ with the following property: $\forall i \geq k \,\, (g(i) \geq p_{\rho}(i))$. Furthermore, since $(\omega^{\gamma}.n)_{n \in \omega}$ converges to $\omega^{\beta_{1}}$, there must be an $l \in \omega$ with $\forall i \geq l \,\, f_{i}(\omega^{\gamma})=f(\omega^{\gamma}.i) \in W_{\rho}$. Taking $a = max\{k,l\}$ we have that, for all $i \geq a$, $f_{i}(\omega^{\gamma})=f(\omega^{\gamma}.i)  \in W_{\rho}$ and $g(i) \geq p_{\rho}$. Therefore, for all $\theta \geq \omega^{\gamma}.a + 1$, $f(\theta) \in W_{\rho} \subset U$. Hence, $(\omega^{\gamma}.a , \omega^{\beta_{1}}] \subset f^{-1}[U]$, and $\omega^{\beta_{1}}$ is an internal point of $f^{-1}[U]$.
	\vspace*{0,25cm}
	
	To verify that $f$ is open, since $f$ is injective, we only need to check open sets of the forms $[0,\delta)$ and $(\delta,\omega^{\beta_{1}}]$. Fix $\delta = \omega^{\gamma}.t + \theta'$. The following holds: 
	\[f[[0,\delta)] = \left[\left(\bigcup\left\{U_{i}: i \in t\right\}\right) \cup f_{t}\left[[1, \theta')\right] \right] \cap ran(f).\] 
	Since $f_{t}$ is a homeomorphism it follows that the set is open in $\ran{f}$. Now, to see that $f\left[(\delta, \omega^{\beta_{1}}]\right]$ is open, consider $f_{t}\left[(\theta' , \omega^{\gamma}]\right]$, the open sets $U_{i}$ for $i > t$, and an $W_{\rho}$ disjoint from the sets $\overline{U_{0}}, \cdots \overline{U_{t+1}}$. The verification is analogous to the ones before.
	
	\vspace*{0,25cm}
	
	We must also consider the case where $\beta_{1}$ is a limit ordinal. In this case, we change the convergent sequence $(\omega^{\gamma}.n : n \in \omega)$ for $(\omega^{\gamma_{n}} : n \in \omega)$, where $(\gamma_{n})_{n \in \omega}$ converges to $\beta_{1}$, which is possible since $\alpha$ is countable. The functions $f_{i}$ will represent copies of $\omega^{\gamma_{i}}+1$ in $U_{i}$, and the function $h$ will be changed to an $h_{n}$ for each corresponding $\omega^{\gamma_{n}}$. Changes should also be made to the $p_{\rho}$'s and  $f$ accordingly. The verification will be almost identical to the one before. 
	
\end{proof}

We present next the revisited proof of Theorem \ref{WKoriginal}, fixing the gap in the original, as we mentioned in the introduction.

\begin{thm}\label{main1}
	Let $X$ be a regular topological space with $X \rightarrow (top \, \omega+1)^{1}_{\omega}$, and $\chi(X) < \mathfrak{b}$. Then $X \rightarrow (top \, \alpha)^{1}_{\omega}$ for all $\alpha < \omega_{1}$.
\end{thm}

\begin{proof}
	Let us consider a partition $X = \bigcup X_{n}$ given by a coloring $f \in \omega^{X}$. For each of the subspaces $X_{n}$, we take the Sequential Cantor-Bendixson decomposition. As in the original proof, we shall consider the following two cases. In the first one, there is an $n \in \omega$, such that $(X_{n})^{Sh(X_{n})}_{S}$ is not empty. In the second one, for all $n \in \omega$, $X_{n}$ is S-scattered. 
	
	\vspace*{0,25cm}
	
	For the first case, we fix $n \in \omega$ given by the hypothesis. Note that, since $(X_{n})^{Sh(X_{n})}_{S}$ is S-perfect, non empty, and the regularity and character are hereditary, we are in the condition of Lemma \ref{Alpha}. Therefore, the monochromatic copies of each $\alpha < \omega_{1}$ are given by the previous result.
	
	\vspace*{0,25cm}
	
	Consider now the second case. First, we note that there must be at least one $j \in \omega$, such that $Sh(X_{j}) \geq \omega_{1}$. Otherwise, we would have $Sh(X_{n})$ countable for all $n \in \omega$, and then $\{SI_{\alpha}(X_{n}) : n \in \omega \mbox{ and } \alpha < Sh(X_{n})\}$ would be a partition of $X$ that contradicts $X \rightarrow (top \, \omega + 1)^{1}_{\omega}$, since from the definition of $SI_{\alpha}(X_{n})$ there is no $ s:\omega \rightarrow SI_{\alpha}(X_{n})$ converging to a point in $SI_{\alpha}(X_{n})$. Fix one such $j \in \omega$. Our goal is to construct the copies of $\alpha$ using the levels of $X_{j}$ to help us.

	\vspace*{0,25cm}
	
	In the proof of Lemma \ref{Alpha}, the hypothesis worked for every point and neighbourhood. This construction, however, is not that simple. It is clear that no point in $SI_{0}(X_{j})$ could represent $\omega$ in a homeomorphism, since they cannot have a sequence in $X_{j}$ converging to them. Therefore, we have to carefully choose the points in each step of the recursive construction. With this in mind, we restate the hypothesis. Suppose that, for all $\beta < \alpha$, there is a $\lambda_{\beta} \leq \omega_{1}$ satisfying that for all $x \in (X_{j})^{\lambda_{\beta}}_{S}$ and open neighbourhood $V$ of $x$, there is an homeomorphism $f:\beta \rightarrow f[\beta] \subset V$ such that if $\beta = \gamma + 1$ and $\gamma$ is a limit ordinal, then $f(\gamma) = x$.

	\vspace*{0,25cm}
	
	We prove that there is an $\lambda_{\alpha} < \omega_{1}$ satisfying the condition in the statement above. Let $\lambda = sup\{\lambda_{\beta} : \beta < \alpha\}$. For all $\beta < \alpha$, we have $\lambda_{\beta} < \omega_{1}$ and $\alpha < \omega_{1}$; hence, $\lambda < \omega_{1}$. Let $\lambda_{\alpha} = \lambda + 1$. Fix $x \in (X_{j})^{\lambda_{\alpha}}_{S}$ and open neighbourhood of $x$ $V$. By Proposition \ref{PropViz}, there is a sequence $s: \omega \rightarrow X_{j}$ converging to $x$, with $s[\omega] \subset (X_{j})^{\lambda}_{S} \cap V$. Since $\lambda > \lambda_{\beta}$, the elements taken for the sequence have the same properties used in Lemma \ref{Alpha}. We can now just repeat the argument used before to obtain the homeomorphism and verify that $\lambda_{\alpha}$ is as desired.
\end{proof}

\section{The principle $\clubsuit_{F}$}\label{Comb}

In this section we define $\clubsuit_{F}$, a variation of the combinatorial principle $\clubsuit$. We will prove that $\diamondsuit^{*}$ implies $\clubsuit_{F}$ and that $\clubsuit_{F}$ is consistent with $\neg CH$.  For the later, we will use an iteration of Cohen forcing using the $CS^{*}$-iteration defined in \cite{fuchino}.
\begin{definition}
	If $X$ is a set of ordinals and  $\beta$ is an ordinal, then we say that $\beta \in \acc(X)$ if and only if $\beta = sup(X \cap \beta)$ and $\beta > 0 $. 
\end{definition}

\begin{definition}
	A sequence $\langle A_{\alpha}^{n} : n < \omega \wedge \alpha < \omega_{1} \rangle$ is called a $\clubsuit_{F}$-sequence if the following holds:
	\begin{itemize}
		\item[(1)] for all $\alpha \in acc(\omega_{1})$ and $n \in \omega$ we have that $A_{\alpha}^{n}$ is an unbounded subset of $\alpha$;
		\item[(2)] for all $f:\omega_{1} \rightarrow \omega$ there are $\alpha \in acc(\omega_{1})$ and $n, m \in \omega$ such that $\alpha \in f^{-1}[\{n\}]$ and $A_{\alpha}^{m} \subset f^{-1}[\{n\}]$.
	\end{itemize}
\end{definition}

\begin{definition}\label{diamondstar}
	We say that $\langle A_{\alpha}^{m} \mid \alpha < \omega_1 \wedge m < \omega \rangle $ is a $\diamondsuit^*$-sequence if and only if for every $X \subseteq \omega_1$ there exists a club $C \subseteq \omega_1$ such that, for all $\alpha \in C$, we have $X \cap \alpha = A_{\alpha}^{m}$, for some $m \in \omega$. 
\end{definition}

\begin{definition}
	Let $\vec{B}=\langle B_{\alpha}^{m} : \alpha < \omega_1 \wedge n < \omega \rangle$ be a $\diamondsuit^*$-sequence.  Consider $\vec{A}$ the sequence given by: $ A_{\alpha}^{n} = B_{\alpha}^{n}$ for all $n \in \omega$ and $\alpha \in \omega_{1}$ such that $B_{\alpha}^{n}$ is unbounded in $\alpha$, and $A_{\alpha}^{n} = \alpha$ otherwise. We say that $\vec{A}$ is the \emph{derived sequence} from $\vec{B}$. \end{definition}

In the next lemma we show an interesting property of a sequence derived from a $\diamondsuit^{*}$-sequence.

\begin{lemma}
	Let $\vec{B}$ be $\diamondsuit^{*}$-sequence. Then $\vec{A}$, the derived sequence from $\vec{B}$, is  a $\clubsuit_{F}$-sequence\footnote{Actually, $\clubsuit_{F}$ can be derived from $\diamondsuit$ but we will not need this fact.  Given $\langle A_{\alpha} \mid \alpha < \omega_1 \rangle$, a $\diamondsuit$-sequence, by \cite[Ch IV, Lemma 2.9 ]{Devlin}, it is possible to obtain a sequence $\langle A_{\alpha}^{m} \mid \alpha < \omega_1 \wedge m \in \omega \rangle$ such that for any $f:\omega_1 \rightarrow \omega$ there exists $n \in \omega$ such that $\langle A_{\alpha}^{m} \mid \alpha \in f^{-1}[\{n\}] \rangle$ is a $\diamondsuit(f^{-1}[\{n\}])$-sequence. Hence $\langle A_{\alpha}^{m} : \alpha < \omega_1 \wedge m \in \omega \rangle$ will be a $\clubsuit_{F}$-sequence. }.
\end{lemma}

\begin{proof}
	We shall prove that $\langle A_{\alpha}^{n} : n \in \omega \wedge \alpha \in \omega_{1} \rangle$ is a $\clubsuit_{F}$-sequence. Indeed, the condition (1) is immediate from our definition. To verify that the other condition is also satisfied let $f:\omega_{1} \rightarrow \omega$ be a function. There exists $n \in \omega$ such that $f^{-1}[\{n\}]$ is stationary and therefore $acc(f^{-1}[\{n\}])$ is a club. Applying the $\diamondsuit^{*}$ property for $f^{-1}[\{n\}]$ we have another club $\mathcal{C}$ such that, for all $\beta \in \mathcal{C}$, there exists $m \in \omega$ with $f^{-1}[\{n\}] \cap \beta = B_{\beta}^{m}$. Finally, for $\beta \in \mathcal{C} \cap acc(f^{-1}[\{n\}])$, $B_{\beta}^{m}$ satisfies $B_{\beta}^{m} = A_{\beta}^{m}$, verifying (2).
\end{proof}

The proof of the main theorem of this section, Theorem \ref{CohenCS}, is a variation of \cite[Theorem3.1]{chen} using the original notation from \cite{fuchino}. We use the Cohen $CS^{*}$-iteration which appears originally in \cite[$\S$ 5]{fuchino}. In what follows we shall define $CS^{*}$-iterations and state some of its properties.

\begin{definition}\cite[Section 4]{fuchino}
	We say that $\mathbb{P}_{\varepsilon} = \langle (\mathbb{P}_{\alpha}, \dot{\mathbb{Q}}_{\alpha}) : \alpha < \varepsilon\rangle$ is a $CS^{*}$-iteration if and only if $\langle (\mathbb{P}_{\alpha}, \dot{\mathbb{Q}}_{\alpha}) : \alpha < \varepsilon\rangle$ is a countable support iteration satisfying the following additional condition: if $\alpha \leq \varepsilon $ and $p < q$ in  $\mathbb{P}_{\alpha}$, then $\diff(p,q) \stackrel{\cdot}{=} \{\beta \in dom(p) \cap dom(q) : p\restriction_{\beta} \ \nVdash p(\beta) = q(\beta)\}$ is finite.
	Furthermore we say that $p \leq^{h}_{\mathbb{P}_{\alpha}} q$ if and only if $p \leq q$ and $p \restriction_{\dom{q}} = q$. We also say that $p \leq^{v}_{\mathbb{P}_{\alpha}} q$ if and only if $ p \leq q$ and $\dom{p}=\dom{q}$. When there is no risk of ambiguity we shall omit $\mathbb{P}_{\alpha}$ from $ \leq^{h}_{\mathbb{P}_{\alpha}}$ and $\leq^{v}_{\mathbb{P}_{\alpha}} $.
\end{definition}

\begin{definition}[Cohen $CS^*$ Iteration]\cite[Section 5]{fuchino} If $  \mathbb{P}_{\varepsilon} = \langle (\mathbb{P}_{\alpha}, \dot{\mathbb{Q}}_{\alpha}) : \alpha < \varepsilon\rangle$ is a $CS^{*}$-iteration such that for every $\alpha < \varepsilon$ we have $\1_{\mathbb{P}_{\alpha}} \forces \dot{\mathbb{Q}}_{\alpha} = Fn(\omega,2)$, then we call $\mathbb{P}_{\varepsilon}$ a Cohen $CS^{*}$-iteration of length $\varepsilon$.
\end{definition}

The next Theorem is the main result of this section. Before proving it, we will derive Corollary \ref{Cor2} and state a few preparatory lemmas that we will need for the proof of Theorem \ref{CohenCS}.

\begin{thm}\label{CohenCS} Let $\mathbb{P}_{\kappa}$ be a Cohen $CS^*$-iteration of length $\kappa$, where $\kappa$ is a regular cardinal $\geq \aleph_2$ such that for every $\alpha < \kappa$ we have $\alpha^{\omega}<\kappa$. Suppose $\vec{B}$ is a $\diamondsuit^{*}$-sequence and  $ \vec{A}=\langle A_{\alpha}^{m} : \alpha < \omega_1 \wedge n \in \omega \rangle $ is the sequence derived from $\vec{B}$. If  $ p \in \mathbb{P}_{\kappa}$, $\sigma$ is a $\mathbb{P}_{\kappa}$-name and $p\forces \sigma:\check{\omega}_{1} \rightarrow \omega$, then there is a condition $p_{*} < p$,  an ordinal $\eta < \omega_1$ and $n,m \in \omega $ such that $$p_* \forces (  \check{A}_{\eta}^{m} \subseteq \sigma^{-1}[\{\check{n}\}] \wedge \sigma(\check{\eta})=\check{n}).$$ In particular  $p_{*} \forces ``\sigma \mbox{ is not a bijection}"$.
\end{thm} 

\begin{cor}\label{Cor2} 
	Assume $\diamondsuit^{*}$ and that $\kappa$ is a regular cardinal $\geq \aleph_2$ such that for all $\alpha < \kappa$ we have $\alpha^{\omega}<\kappa$. Let $\vec{A}$ be the sequence derived from a $\diamondsuit^*$-sequence $\vec{B}$. Then $$\forces_{\mathbb{P}_{\kappa}} \vec{A} \text{ is a } \clubsuit_{F} \text{-sequence}.$$
\end{cor}

\begin{proof}
	Let $ \sigma $ be a $\mathbb{P_{\kappa}}$-name and $q \in \mathbb{P}$ such that $q \forces \sigma:\omega_1 \rightarrow \omega$. Fix $p \in \mathbb{P}$ such that $p \leq q $. Applying Theorem \ref{CohenCS} to $p$ and $\sigma$ it follows that there exists $p_* \leq p$, $\eta < \omega_1$ and $n,m \in \omega$ such that 
	$$p_* \forces \check{A}_{\eta}^{m} \subseteq \sigma^{-1}[\{\check{n}\}] \wedge \sigma(\check{\eta})=\check{n}.$$
	Therefore $$\forces_{\mathbb{P}_{\kappa}} \vec{A} \text{ is a } \clubsuit_{F} \text{-sequence}.$$
\end{proof}

Now we focus on what is needed to prove Theorem \ref{CohenCS}. From here until the end of this Section the forcing $\mathbb{P}_{\alpha}$ will denote a Cohen $CS^{*}$-iteration of length $\alpha$ for any ordinal $\alpha$. The following results are central to our proof. 

\begin{lemma}\cite[Lemma 4.3]{fuchino}
	Let $\gamma$ be an ordinal, and suppose that $\langle p_{n} : n \in \omega \rangle $ is a sequence of elements of $\mathbb{P}_{\gamma}$ such that $m < l < \omega $ implies $p_{l} \leq^{h} p_{n}$. Then $r= \bigcup_{n \in \omega} p_{n} \in \mathbb{P}_{\gamma}$. 
\end{lemma}

\begin{cor}
	Let $\alpha < \omega_1$, and let $\langle p_{\beta} : \beta < \alpha \rangle $ be a sequence of elements of $\mathbb{P}_{\kappa}$ such that $\beta < \gamma < \alpha$ implies $p_{\beta} \leq^{h} p_{\gamma} $. Then $\bigcup_{\beta} p_{\beta} $ is the greatest lower bound of the sequence $\langle p_{\beta} : \beta < \alpha \rangle $. \end{cor}

\begin{lemma}\cite[Lemma 5.1]{fuchino}\label{decide}
	Given an ordinal $\theta$ and $p, q \in \mathbb{P}_{\theta}$, if $p \leq q$, then there is $p' \leq p$ such that, for all $\alpha \in \diff(p',q)$, $p'\restriction_{\alpha}$ decides the value of $p'(\alpha)$.
\end{lemma}

\begin{lemma}\label{rename}
	Let $\theta$ be an ordinal, $p \in \mathbb{P}_{\theta}$, and $\langle \sigma_{\alpha} : \alpha \in \dom{p}  \rangle$ be a sequence of $\stackrel{\cdot}{\mathbb{Q}}_{\alpha}$-names. If for all $\alpha \in \dom{p}$ we have $$p \restriction_{\alpha} \forces p(\alpha) = \sigma_{\alpha},$$ then for any formula $\varphi(x)$ and $\mathbb{P}_{\kappa}$-name $\tau$ we have: $$p \forces \varphi(\tau)$$ if and only if 
	$$ p' \stackrel{\cdot}{=} \langle (\alpha,\sigma_{\alpha}) : \alpha \in \dom{p} \rangle  \forces \varphi(\tau).$$
\end{lemma}

\begin{proof}
	Given $p, p' \in \mathbb{P}_{\theta}$ as in the hypothesis we prove the following statement by induction on $\gamma \leq \theta$:
	
	$$(\triangle)_{\gamma} \ \forall r (r \in \mathbb{P}_{\gamma} \rightarrow ( r \leq p \restriction \gamma \leftrightarrow r \leq p' \restriction \gamma ))$$
	
	Notice that $(\triangle)_{\gamma}$ implies that $p\restriction \gamma$ and $p'\restriction \gamma$ forces the same statements.
	
	Fix $\gamma \leq \theta $ and suppose $r \leq p\restriction \gamma$, then $r \leq p'\restriction \gamma$. Indeed, we have $r \restriction_{ \alpha} \forces r(\alpha) \leq p(\alpha)$, since $r \restriction_{\alpha} \leq p \restriction_{\alpha}$. Therefore $r \restriction_{ \alpha } \forces r(\alpha) \leq p(\alpha) = p'(\alpha)$. We need to see now that $\diff(r,p'\restriction_{\gamma})$ is finite. If $\alpha \in \diff(r, p'\restriction_{\gamma})$ then $r \restriction_{ \alpha }\not\forces r(\alpha) = p'(\alpha)$. Since $r \restriction_{ \alpha }\forces p(\alpha) = p'(\alpha)$, $r \restriction_{ \alpha} \not\forces r(\alpha) = p(\alpha)$. Therefore $\alpha \in \diff(r,p\restriction_{\gamma})$.  
	
	Let $ r \leq p' \restriction_{\gamma}$. If $ \zeta  < \gamma $, then from our induction hypothesis $(\triangle)_{\zeta}$ we have that $p'\restriction_{\zeta} \forces p(\zeta) = \sigma_{\zeta}$. Therefore $r \leq p' \restriction_{\gamma}$ implies $r\restriction \zeta \forces r(\zeta) \leq \sigma_{\zeta} = p(\zeta)$. Thus $r \leq p \restriction_{\gamma}$.  Now fix $\alpha \in \diff(r,p\restriction_{\gamma})$, then $r\restriction_{\alpha} \not \forces p(\alpha) = r(\alpha)$. Since $\alpha < \gamma$ we have $r\restriction_{\alpha} \forces p(\alpha) = p'(\alpha)$, it follows that $r \restriction_{\alpha} \not\forces p'(\alpha) = r(\alpha)$. Hence $\alpha \in \diff(r,p'\restriction_{\gamma})$ and  $(\triangle)_{\gamma}$ holds.
	
	Now $(\triangle)_{\theta}$ implies the lemma.

\end{proof}

We are finally ready to prove Theorem \ref{CohenCS}.

\begin{proof}[\textbf{Proof of Theorem \ref{CohenCS}}]
	Let $\sigma$ be a name for a given $f:\omega_{1} \rightarrow \omega$ in the extension.
	We start by taking any $p \in \mathbb{P}$ such that $p \forces \sigma: \check{\omega}_1 \rightarrow \omega$. Then we will show that there is a condition $p_{*} < p$, $\eta < \omega_1$ and $n,m \in \omega $ such that
	
	\begin{gather*} 
		p_* \forces  (  \check{A}_{\eta}^{m} \subseteq \sigma^{-1}[\{\check{n}\}] \wedge \sigma(\check{\eta})=\check{n})
	\end{gather*}

	Following \cite{chen} we shall construct a  $\omega_{1}$ sequence that decides the evaluation of $\sigma$ on each $\alpha < \omega_1$. Then we use a pressing down argument to find a stationary set in $V$ where we can apply a $\clubsuit_{F}$-sequence derived from a $\diamondsuit^*$-sequence in order to find the condition $p_*$.

	We shall construct inductively  $\langle q_{\alpha} \mid \alpha < \omega_1\rangle$, $\langle n_{\alpha} \mid \alpha < \omega_1 \rangle$, together with an auxiliary sequence $\langle p_{\alpha} : \alpha < \omega_{1} \rangle$. These three sequences should satisfy the following conditions for all $\alpha < \omega_{1}$:
	
	\begin{enumerate}
		\item $\langle p_{\gamma} : \gamma < \alpha \rangle$ is a sequence such that $ \gamma < \gamma' < \alpha$ implies $p_{\gamma'} \leq^{h} p_{\gamma}$;
		\item $q_{\alpha} \leq^{v} p_{\alpha}$ and $q_{\alpha} \forces \sigma(\alpha) = \mbox{\v{n}}_{\alpha}$;
		\item $u_{\alpha} = \diff(q_{\alpha},p_{\alpha}) \subseteq \dom{p} \cup \bigcup_{\gamma < \alpha } \dom{q_{\gamma}}$;
		\item  $q_{\alpha} \restriction_{ (dom(p_{\alpha}) \setminus u_{\alpha} )}= p_{\alpha} \restriction_{(dom(p_{\alpha}) \setminus u_{\alpha} )}$;
		\item $q_{\alpha} \restriction_{u_{\alpha}} \in Fn(\kappa, \{\check{t} : t \in Fn(\omega,2)\})$.
	\end{enumerate}
	
	Suppose we already constructed $\langle p_{\beta} :\beta < \alpha \rangle $, $ \langle q_{\beta} : \beta < \alpha \rangle $, and $ \langle n_{\beta} : \beta < \alpha \rangle $. 
	
	Let $r_{\alpha} = p$ if $\alpha = 0$ or $r_{\alpha} = \bigcup_{\beta \in \alpha} p_{\beta}$ otherwise. Fix $w_{\alpha} \leq r_\alpha$ such that there is $n_{\alpha} \in \omega$ satisfying $w_{\alpha} \forces \sigma(\alpha) = \check{n}_{\alpha}$. Applying lemma \ref{decide} on $r_{\alpha}$ and $w_{\alpha}$, we obtain $q^{*}_{\alpha} \leq w_{\alpha}$ such that, for all $\gamma \in \diff(q^{*}_{\alpha},r_{\alpha})$, there exists $\check{t}_{\gamma}$ satisfying $q^{*}_{\alpha}\restriction_{ \gamma} \forces q^{*}_{\alpha}(\gamma)=\check{t}_{\gamma} $. \\
	
	Define the following:
	
	\begin{itemize} 
		\item $p_{\alpha} = r_{\alpha}  \cup (q^{*}_{\alpha} \restriction_{dom(q^{*}_{\alpha}) \setminus dom(r_{\alpha})})$; 
		\item $d_{\alpha} = \diff(q^{*}_{\alpha},p_{\alpha})$; 
		\item $q_{\alpha} = p_{\alpha} \restriction_{\dom{p_{\alpha}} \setminus d_{\alpha}} \cup \,\, \{\check{t}_{\gamma} : \gamma \in  d_{\alpha} \}$.
	\end{itemize} 
	
	Let us verify that $p_{\alpha}$ and $q_{\alpha}$ satisfy conditions $(1)-(5)$.\\
	
	First we note that $p_\alpha \leq^{h} r_{\alpha}$, and therefore $p_\alpha \leq^{h} p_{\beta}$ for all $\beta < \alpha$. Indeed, $\dom{r_{\alpha}} \subseteq \dom{p_{\alpha}}$ and for all $\gamma \in \dom{r_{\alpha}}$ we have $p_{\alpha} \restriction_{\gamma} \forces r_{\alpha}(\gamma) = p_{\alpha}(\gamma)$ since they are the same name. Furthermore $\diff(p_{\alpha},r_{\alpha}) = \emptyset$ and, by the definitions given, the inequality holds.\\
	
	Next, we verify (2).  Note that  for all $\gamma \in \dom{q^{*}_{\alpha}} \setminus d_{\alpha} $ we have $$q^{*}_{\alpha} \restriction_{\gamma} \forces q^{*}_{\alpha}(\gamma) = p_{\alpha}(\gamma) = q_{\alpha}(\gamma),$$ and for all $\gamma \in d_{\alpha}$ we have 
	$$q^{*}_{\alpha} \restriction_{\gamma} \forces q^{*}_{\alpha}(\gamma) = \check{t}_{\gamma} = q_{\alpha}(\gamma).$$ It follows from Lemma \ref{rename}  that $q_{\alpha}$ forces the same statements that $q^{*}_{\alpha}$ forces. In particular, $q_{\alpha} \forces \sigma(\alpha) = \mbox{\v{n}}_{\alpha}$.\\
	
	From the definition of $p_{\alpha}$ and $q_{\alpha}$, it follows that $\dom{q^{*}_{\alpha}} = \dom{p_{\alpha}} = \dom{q_{\alpha}}$. Hence, in order to verify $q_{\alpha} \leq^{v} p_{\alpha}$, we only have to verify that $q_{\alpha} \leq p_{\alpha}$.  Consider $\zeta \in dom(p_{\alpha})$. If $\zeta \in dom(r_{\alpha})$, then $p_{\alpha}(\zeta) = r_{\alpha}(\zeta)$ and $q_{\alpha}\restriction_{ \zeta } \forces p_{\alpha}(\zeta) = r_{\alpha}(\zeta) \geq q^{*}_{\alpha}(\zeta) = q_{\alpha}(\zeta) $, since $q^{*}_{\alpha}\restriction_{\zeta}$ also forces it. \\
	
	If $\zeta \in dom(p_{\alpha}) \setminus dom(r_{\alpha}) $, since $d_{\alpha} \subseteq dom(r_{\alpha})$, then $p_{\alpha}(\zeta) = q^{*}_{\alpha}(\zeta)$. Therefore $$q_{\alpha}\restriction_{\zeta} \ \forces  q_{\alpha}(\zeta) =  q^{*}_{\alpha}(\zeta) =  p_{\alpha}(\zeta) .$$ \\
	
	Now we verify condition (3). By Lemma \ref{rename}, $\gamma \in \diff(q_{\alpha},p_{\alpha})$ if and only if $\gamma \in \diff(q^{*}_{\alpha},p_{\alpha})$. Therefore $u_{\alpha} = d_{\alpha}$. Furthermore, from our induction hypothesis $$u_{\alpha} \subset \dom{r_{\alpha}} \subseteq  dom(p) \cup \bigcup_{\gamma < \alpha } \dom{p_{\gamma}} = dom(p) \cup \bigcup_{\gamma < \alpha } \dom{q_{\gamma}}$$\\
	
	Condition (4) and (5) follows directly from our definitions of $q_{\alpha}$  and the fact observed above that $u_{\alpha}=d_{\alpha}$.\\

	Consider now a bijection $\Phi: \bigcup_{\beta < \omega_{1}} dom(q_{\beta}) \rightarrow \omega_{1}$, and let $Y \subseteq \omega_{1}$ be a club such that, for all $\alpha \in Y$, $$\Phi(\bigcup_{\beta<\alpha} dom(q_{\beta})) \subseteq \alpha.$$ Therefore, for all $\alpha \in Y$, we have $a_{\alpha}:=\Phi[u_{\alpha}] \subseteq \alpha$.\\ 
	
	Let $Y_0 \subseteq Y$ be a stationary set and $k \in \omega$ be such that for all $\alpha \in Y_0, \,\, |a_{\alpha}|=k$. Let $\phi_{0}:Y_0 \rightarrow \omega_1$ be a regressive function given by $\phi_{0}(\alpha)= min(a_{\alpha})$. Applying Fodor's Lemma, we find $Y_{1} \subseteq Y_{0}$ stationary such that $\phi_{0}$ is constant on $Y_{1}$. Recursively, for $n < k $, we construct $\phi_{n}:Y_{n} \rightarrow \omega_1$ such that $\phi_{n}(\alpha)$ is the $n^{th}$ element of $a_{\alpha}$. After $k$ iterations we find $Y_{k+1}$ stationary and $a \in \omega_{1}^{<\omega}$ such that $\alpha \in Y_{k+1}$ implies $a_{\alpha}=a$.\\
	
	For all $\alpha \in Y_{k+1}$ we have $u_{\alpha}=u$ for a fixed $u \in \kappa^{< \omega}$. The set $$W = \{ r \in Fn(\kappa,Fn(\omega,2)) : dom(r) = u \}$$ is countable, therefore there is $S \subseteq Y_{k+1}$ stationary and $r \in W$ such that for all $\alpha \in S$ and $\gamma \in u$ we have $q_{\alpha}(\gamma)=\check{r}(\gamma)$.\\
	
	Fix $n^* \in \omega$ such that $T = \{ \alpha \in S : n_{\alpha}=n^* \}$ is stationary. Using that $\vec{B}$ is a $\diamondsuit^{*}$-sequence, we can find a club $C$ such that, for every $\alpha \in C $, there exists $m \in \omega $ such that $ B_{\alpha}^{m} = T \cap \alpha$. Finally, let $ \eta \in C \cap T \cap acc(T)$, then $sup(B_{\alpha}^{m}) = \eta$ and $B_{\eta}^{m} = A_{\eta}^{m}$.\\
	
	Next, consider $\xi \in T$ such that  $\xi > \eta$. Let $p_*=q_{\xi}$. We shall verify that that $p_*$ is the condition we are looking for.\\
	
	First we shall see that if $\alpha \in S \cap \xi$, then $q_{\xi} \leq^{h} q_{\alpha}$. For every $\alpha \in Y$ it holds that $q_{\alpha} = p_{\alpha} \restriction_{(dom(p_{\alpha} \setminus d_{\alpha})}$. 
	We know that $p_{\xi} \leq^{h} p_{\alpha}$ and $u_{\alpha}=d_{\alpha}$, therefore $p_{\alpha} \restriction_{(dom(p_{\alpha} \setminus d_{\alpha})} = p_{\xi}  \restriction_{(dom(p_{\alpha} \setminus d_{\alpha})}$.
	Hence $q_{\alpha}\restriction_{(dom(p_{\alpha} \setminus d_{\alpha})} = q_{\xi} \restriction_{(dom(p_{\alpha} \setminus d_{\alpha})}$. From $\alpha \in S$ we have that $u_{\alpha} = u = u_{\xi}$ and  $q_{\alpha} \restriction_{ u_{\alpha}} = q_{\xi} \restriction_{u}$. Thus $q_{\xi} \leq^{h} q_{\alpha}$ and we have $q_{\xi} \forces \sigma \restriction_{ S \cap \xi} = \langle n_{\alpha} \mid \alpha \in S \cap \xi \rangle $.
\end{proof}

Next, given a cardinal $\kappa$ as in the hypothesis of Theorem \ref{CohenCS}, we will address what happens with $\mathfrak{d}$ and $2^{\aleph_0}$ in any $\mathbb{P}_{\kappa}$ extension. We will need these results in Section \ref{Top}, where we also obtain that $\mathfrak{b}=\omega_{1} < 2^{\aleph_0}$ in a $\mathbb{P}_{\kappa}$ extension as application of Theorem \ref{main1} (see Corollary \ref{b1}). \\

We first observe that, as in the usual Cohen forcing, under certain assumptions, our extension will have $\omega_{1} < \mathfrak{d}$. In order to verify that  we will use the two following Lemmas:

\vspace*{0.25cm}

\begin{lemma}\cite[Lemma 5.4]{fuchino} \label{kcc} Suppose that, for all $\alpha < \kappa$, we have $\alpha^\omega < \kappa$ and $\kappa$ is a regular cardinal $> \omega_1$. Then the forcing $\mathbb{P}_{\kappa}$ has the $\kappa$-cc property. 
\end{lemma}
\begin{proof}
	We are in the hypothesis of the $\Delta$-system lemma used in \cite{fuchino}. The proof of \cite[Lemma 5.4]{fuchino} also holds here. 
\end{proof}

\vspace*{0.25cm}

\begin{lemma}\label{dmaiorkappa}
	Let $ \kappa $ be a regular cardinal such that, for all $\alpha < \kappa$, we have $\alpha^{\omega} < \kappa$. Let $\mathbb{P}_{\kappa}$ be the a $CS^{*}$-iteration of Cohen forcing of length $\kappa$. Then,  in any $\mathbb{P}_{\kappa}$-generic extension, it holds that $\mathfrak{d}\geq \kappa$.  
\end{lemma}

\begin{proof}
	Let $V[G]$ be a $\mathbb{P}_{\kappa}$-generic extension of $V$. Consider $\mathcal{F} \in V[G]$ a dominating family of size $\mathfrak{d}$. By contradiction, suppose that $\mathfrak{d} < \kappa$. In $V[G]$, let $\mathcal{F}=\{f_{\alpha} \mid \alpha < \mathfrak{d} \}$ and let $\Phi:\mathfrak{d}\times \omega \rightarrow \omega$ such that $\Phi(\alpha,m)=f_{\alpha}(m)$ for each $\alpha < \mathfrak{d}$ and $m \in \omega$.
	
	Then $\Phi  \subseteq \mathfrak{d} \times \omega \times \omega$ codes $\mathcal{F}$. Using that $\mathbb{P}_{\kappa}$ is $\kappa$-cc by standard forcing arguments it follows that $\Phi \in V[G \restriction \xi]$ for some $\xi < \kappa$, and consequently $\mathcal{F} \subseteq V[G \restriction \xi]$.

	Let $ x_{\xi+1}$ be the Cohen real added at step $\xi+1$. It follows that $x_{\xi+1}$ can not be dominated by any real in $V[G \restriction \xi] \supseteq \mathcal{F}$, contradicting our hypothesis that $\mathcal{F}$ is a dominating family.
	
	Thus $\mathfrak{d} \geq \kappa$ in $V[G]$. 
	
\end{proof}

\begin{cor}\label{omega1} Suppose $\diamondsuit^{*}$ holds \footnote{We included $\diamondsuit^{*}$ as in the hypothesis of Theorem \ref{CohenCS}, but for Corollary \ref{omega1} we do not need that. We need that $\omega_{1}$ is not collapsed, which follows, for example, from \cite[Corollary 5.3]{fuchino}. }.  Let $ \kappa $ be a regular cardinal such that, for all $\alpha < \kappa$, we have $\alpha^{\omega} < \kappa$. Let $\mathbb{P}_{\kappa}$ be the a $CS^{*}$-iteration of Cohen forcing of length $\kappa$. Then $1_{\mathbb{P}_{\alpha}} \forces \omega_{1} < \mathfrak{d}$.
\end{cor}

\begin{proof} 
	From Theorem \ref{CohenCS} we have that $1_{\mathbb{P}_{\kappa}} \forces \omega_{1} = \check{\omega}_{1}$. By Lemma \ref{dmaiorkappa} we have $1_{\mathbb{P}_{\kappa}} \forces \kappa \leq \mathfrak{d}$. Thus  $1_{\mathbb{P}_{\kappa}} \forces \omega_1 < \kappa \leq \mathfrak{d}$.
\end{proof}

\begin{cor}
	$Con(ZFC) \rightarrow Con(ZFC + \clubsuit_{F} + \neg CH)$
\end{cor}

\begin{remark}\label{continuum}
	Under the same hypothesis of Lemma \ref{dmaiorkappa}, by the usual counting of nice names arguments for subsets of $\omega$, if  $\kappa^{<\kappa} = \kappa$, then $\kappa = \mathfrak{c}$.
\end{remark}

We conclude this section with the following question: 

\begin{per}
	Does $\clubsuit_{F}$ imply $\clubsuit$ or vice-versa?
\end{per}

\section{The bound on the character of $X$}\label{Top}

Since  $\chi (X) < \mathfrak{b}$ appears in the hypothesis of Theorem \ref{main1}, it becomes relevant to ask whether $\mathfrak{b}$ really is the best bound. A construction made in  \cite[Theorem 4]{WeissKomjath} shows us that this is the case under $\diamondsuit$.\\

The main thms of this section are Theorem \ref{exampleclub} and Theorem \ref{Cor1}. They both use our results from the previous section and show the consistency of the existence of a topological space $(X,\tau)$ such that $X \rightarrow (top \, \omega+1)^{1}_{\omega}$, $X \nrightarrow (top \, \omega^{2}+1)^{1}_{\omega}$, and $\chi(X) = \omega_{1} < \mathfrak{c}$. Also, together with  Theorem \ref{main1}, they also verify that the principle $\clubsuit_{F}$ and the existence of such topological space are consistent with $\mathfrak{b} = \omega_{1} < \mathfrak{c}$.\\

In Theorem \ref{exampleclub} we will change the construction from \cite[Theorem 4]{WeissKomjath}, so we can get it using $\clubsuit_{F}$ instead of $\diamondsuit$. This will give us a space with almost the same properties as the space obtained in \cite[Theorem 4]{WeissKomjath}, with the exception that the character of the space in Theorem \ref{exampleclub} will be bounded by $\mathfrak{d}$ instead of $\omega_1$ (see Lemma \ref{character}).

\begin{thm}\label{exampleclub}
	Suppose $\langle A_{\alpha}^{m} \mid \alpha < \omega_{1} \wedge m \in \omega \rangle $ is a $\clubsuit_{F}$-sequence. Then there is a regular topological space $(\omega_{1},\tau)$ such that $(\omega_1,\tau) \rightarrow (\omega+1)^{1}_{\omega}$, $(\omega_{1},\tau) \nrightarrow (\omega^{2}+1)^{1}_{\omega}$ and for every limit ordinal $\alpha \in \omega_{1}$ and for every $m \in \omega$ there exists and increasing sequence of ordinals in $ A_{\alpha}^{m}$ converging to $\alpha$.    
\end{thm}

\begin{proof}
	
	Let $\langle A_{\alpha}^{n} : n \in \omega \,\, \alpha \in \omega_{1} \rangle$ be a $\clubsuit_{F}$-sequence and $\alpha \in acc(\omega_{1})$. Let $\gamma_{n}$ be a  strictly increasing sequence converging to $\alpha$. Define $a_{0} = \{t_{0}(0)\}$ where $t_{0}(0)$ is the first element of $A_{\alpha}^{0}$. Given $n \in \omega$, suppose that, for all $m < n$, $a_{m} = \{t_{m}(0), \cdots , t_{m}(m)\}$ is already defined. Let $a_{n} = \{t_{n}(0), \cdots , t_{n}(n)\}$ be given by the following: $t_{n}(j)$ is the least element of $A_{\alpha}^{j}$ that is greater than $\gamma_{n}$ and all $t_{k}(l)$ constructed beforehand. Notice that this is possible since all $A_{\alpha}^{n}$ are unbounded in $\alpha$.
	
	Order the ordinals $t_{k}(l)$ as a strictly increasing sequence $s(n)$ converging to $\alpha$. Now, for each $n \in \omega$, if $s(n)$ is a limit ordinal, let $(\beta_{i}^{n})_{i \in \omega}$ be a strictly increasing sequence converging to $s(n)$. Otherwise consider $\beta_{i}^{n} = s(n)-1$. We shall refine the topology in $\omega_{1}$ by considering new neighbourhoods of $\alpha$ given by $N_{\alpha}(h,p) = \{\alpha\} \cup (\bigcup\{(\beta_{h(m)}^{m}, s(m)] : m \geq p\})$, where $h \in \omega^{\omega}$ and $p \in \omega$.
	
	As in the construction \cite{WeissKomjath}, our space is regular and it cannot contain any copy of $\omega^{2}+1$ because of the new topology. Then, we just have to verify that $X \rightarrow (top \, \omega + 1)^{1}_{\omega}$. For that we fix any coloring $f:\omega_{1} \rightarrow \omega$. Using (2) of $\clubsuit_{F}$ there are $m,n \in \omega$ and $\alpha \in \omega_{1}$ such that $\alpha \in f^{-1}[\{n\}]$ and $A_{\alpha}^{m} \subset f^{-1}[\{n\}]$. Now, by construction, the elements $\gamma \in A_{\alpha}^{m}$ that are in the new neighbourhood of $\alpha$ together with $\alpha$ are the monochromatic copy of $\omega + 1$.
\end{proof}

\vspace*{0,25cm}

One possible way to verify that there is a space as in Theorem \ref{exampleclub} with character $\omega_{1} = \mathfrak{b} < \mathfrak{c}$ would be to prove the consistency of $\clubsuit_{F} + \neg CH + \mathfrak{b} = \mathfrak{d}$. We still do not know the answer for the following: 

\begin{per}
	$Con(ZFC) \rightarrow Con(ZFC + \clubsuit_{F} + 2^{\aleph_0} > \omega_1 + \mathfrak{d}=\mathfrak{b})$?
\end{per}

Nevertheless, in the next theorem we show that in the generic extensions given by Theorem \ref{CohenCS}, there exists a space with the properties mentioned above. 

\begin{thm} \label{Cor1}
	Assume $\diamondsuit^{*}$ and that $\kappa $ is a regular cardinal such that $\kappa \geq \aleph_2$ and, for every $\alpha < \kappa$, we have $\alpha^{\omega}<\kappa$. Let $\mathbb{P}_{\kappa}$ be the Cohen $CS^{*}$-iteration of length $\kappa$ and $G$ be a $\mathbb{P}_{\kappa}$-generic. Then, in $V[G]$, there exists a topological space $(X,\tau)$ such that $X \rightarrow (top \, \omega+1)^{1}_{\omega}$,  $X \nrightarrow (top \, \omega^{2}+1)^{1}_{\omega}$ and $\chi(X)=\omega_1 < \mathfrak{c}$.
\end{thm}

\begin{proof}
	Let $\vec{A}=\langle A_{\alpha}^{m} : \alpha < \omega_1 \wedge n \in \omega \rangle $ be the $\clubsuit_{F}$-sequence derived from a $\diamondsuit^{*}$-sequence $\vec{B}$ in $V$. Let $(X,\bar{\tau})$ the space obtained by applying, in V,  Theorem \ref{exampleclub} to  $\vec{A}$. Let  $(X,\tau)$ be the topological space generated, in $V[G]$, using $\bar{\tau}$ as a basis. We will prove that $(X,\tau)$ is the space we wanted. Let $p \in \mathbb{P}_{\kappa} $ and let $\sigma$ be a $\mathbb{P}_{\kappa}$-name such that \[p \forces \sigma: \omega_{1} \rightarrow \omega.  \] Recall that from Theorem \ref{CohenCS} we have $\forces_{\mathbb{P}_{\kappa}} \check{\omega_{1}} = \omega_1 $. We will find $q \leq p $, $\dot{S}$ and $n$ such that $$q \forces \sigma[\dot{S}] = \{n\} \wedge \dot{S} \subseteq \omega_1 \wedge   ``(\dot{S},\tau) \text{ is homeomorphic to } (\omega + 1, \in)"  .$$ 
	
	We apply Theorem \ref{CohenCS} to $\vec{A}$, $\sigma$ and $p$, to obtain $p_*$, $\eta < \omega_1$ and $n,m \in \omega$ such that 
	\begin{gather*} p_* \forces (  \check{A}_{\eta}^{m} \subseteq \sigma^{-1}[\{\check{n}\}] \wedge \sigma(\check{\eta})=\check{n})\end{gather*}
	
	Next, we verify that $p_*$ is the condition $q$ that we seek.\\
	
	Let $ \vec{w} =\{w_t : t \in \omega \}$ be the sequence given by Theorem \ref{exampleclub} such that $\vec{w} \subseteq  A_{\eta}^{m}$ and $\vec{w} \nearrow \eta$. Let $\dot{S}= \check{w} \cup\{\check{\eta}\}$ and notice that $(w \cup \{\eta\},\bar{\tau}) \mbox{ is homeomorphic to } (\omega+1, \in)$. Therefore $$p_* \forces ``(\dot{S},\tau) \mbox{ is homeomorphic to } (\omega+1, \in)".$$  We also have $p_* \forces \dot{S} \subseteq \check{A}^{m}_{\eta} \cup \{\eta\} \subseteq \sigma^{-1}[\{n\}].$
	Since $CH$ holds in $V$ we have $\chi(X,\tau)=\omega_1$ in $V$ and therefore $\chi(X,\tau)=\omega_1$ in $V[G]$. From Corollary \ref{omega1} we have $\omega^{V[G]}_{1} < \mathfrak{c}$. Thus $(X,\tau)$ is the space we wanted.
\end{proof}

\begin{cor}\label{ConCor1}
	Assume the consistency of $ZFC$. Then $ZFC$ is consistent with the existence of a topological space $(X,\tau)$ such that $X \rightarrow (top \, \omega+1)^{1}_{\omega}$,  $X \nrightarrow (top \, \omega^{2}+1)^{1}_{\omega}$ and $\chi(X)=\omega_1 < \mathfrak{c}$.
\end{cor}
\begin{proof}
	If $V=L$, then $\diamondsuit^{*}$ holds, and if $\kappa \geq \aleph_{2}$ is a regular cardinal, then, for all $\alpha < \kappa$, we have $\alpha^{\omega} < \kappa$. Working in $L$, let $\kappa \geq \aleph_2 $ be a regular cardinal. By Theorem \ref{CohenCS}, if $G$ is a $\mathbb{P}_{\kappa}$-generic then $V[G]$ is a model of $ZFC$ where, by Theorem \ref{Cor1}, there exists $(X,\tau)$ as in the hypothesis of the corollary. \end{proof}

Now, considering $(X,\tau)$ as in Theorem \ref{Cor1}, applying Theorem \ref{main1}, it follows that  $\mathfrak{b} \leq \chi(X)$. Therefore we must have the following:

\begin{cor} \label{b1} Suppose $\diamondsuit^{*}$ holds, $\kappa \geq \aleph_2$ is a regular cardinal, and, for all $\alpha < \kappa$, $\alpha^{\omega} < \kappa$. Then in any $\mathbb{P}_{\kappa}$-generic extension $\mathfrak{b} = \omega_{1}$ .
\end{cor}

Let $(Y,\tau)$ be a space obtained by applying Theorem \ref{exampleclub} to a $\clubsuit_{F}$-sequence $\vec{A}$. It is easy to see that $\chi(Y) \leq \mathfrak{d}$. In the following lemmas we will present sufficient conditions on $\vec{A}$ which imply that the character of $(Y,\tau)$ is $\mathfrak{d}$.

\begin{lemma}\label{character}
	If $\vec{A}$ is a $\clubsuit_{F}$-sequence such that there exists  a limit ordinal $\alpha \in \omega_{1}$ and $m \in \omega$ satisfying $A_{\alpha}^{m} \subset acc(\omega_{1})$, then $\chi(Y) = \mathfrak{d}$. 
\end{lemma}

\begin{proof}
	We have $\chi(Y) \leq \mathfrak{d}$ since the set $\{N_{\alpha}(h,n) : h \in \mathcal{D} \wedge n \in \omega \wedge \alpha \in acc(\omega_{1})\}$, for any given dominating family $\mathcal{D} \subset \omega^\omega$, together with the isolated successors is a basis for $Y$.
	
	Now, assume that $\chi(Y) < \mathfrak{d}$, and take $\alpha$ and $m$ as in the hypothesis of the lemma. Let $\mathcal{B}$ be a base of basic open sets for $\alpha$ of size $< \mathfrak{d}$. We write $\mathcal{B} = \{N(g_{\theta}, n_{\theta}) : \theta < \lambda\}$ for some $\lambda < \mathfrak{d}$. Take $J = \{r \in \omega : s(r) \in A_{\alpha}^{m}\}$ which is infinite by construction. Let $\phi: J \rightarrow \omega$ be the increasing bijection. Note that $\{g_{\theta}\restriction_{J} \circ \phi^{-1} : \theta < \lambda\}$ is a dominating family, which is a contradiction.
	
	Given $f \in \omega^{\omega}$ consider $h : \omega \rightarrow \omega$ defined by $h(i) = f \circ \phi(i)$ if $i \in J$ and $h(i) = 0$ otherwise. $N(h, 0)$ is a neighbourhood for $\alpha$ so there must be $\gamma < \lambda$ such that $N(g_{\gamma},n_{\gamma}) \subset N(h,0)$. If $i \in J \setminus n_{\gamma}$, since $s(i)$ is a limit ordinal, we must have $\beta_{h(i)}^{i} < \beta_{g_{\gamma}(i)}^{i}$ and therefore $h(i) < g_{\gamma}(i)$. This yields $g_{\gamma}\restriction_{J} \circ \phi ^{*}\geq f $ since $h\restriction_{J} = f \circ \phi$.
\end{proof}

\begin{lemma}\label{clubfromdiam}
	If $\vec{A}$ is a $\clubsuit_{F}$-sequence derived from a $\diamondsuit^{*}$-sequence $\vec{B}$, then there exists a limit ordinal $\alpha \in \omega_{1}$ and $m \in \omega$ such that $A_{\alpha}^{m} \subset acc(\omega_{1})$.
\end{lemma}
\begin{proof}
	Use $\diamondsuit^{*}$ on the set $acc(\omega_{1})$ to find $\mathcal{C}$ the club that guesses this set. Now, if $\alpha \in \mathcal{C}$ is a limit of limit ordinals, then there is  $m \in \omega$ such that $B_{\alpha}^{m} = acc(\omega_{1}) \cap \alpha$. Therefore $B_{\alpha}^{m} = A_{\alpha}^{m} \subset acc(\omega_{1})$.
\end{proof}

The following corollary contrasts with the construction made in Theorem \ref{Cor1}. Here we first consider the extension by a Cohen $CS^{*}$-iteration and afterwards construct the topological space using a $\clubsuit_{F}$-sequence following Theorem \ref{exampleclub}.

\begin{cor}\label{Yd}
	Let $\mathbb{P}_{\kappa}$ be a Cohen $CS^*$-iteration of length $\kappa$, where $\kappa$ is a regular cardinal $\geq \aleph_2$ such that for every $\alpha < \kappa$ we have $\alpha^{\omega}<\kappa$. Suppose $\vec{B}$ is a $\diamondsuit^{*}$-sequence and  $ \vec{A}=\langle A_{\alpha}^{m} : \alpha < \omega_1 \wedge n \in \omega \rangle $ is the sequence derived from $\vec{B}$. 
	If $G$ is a $\mathbb{P}_{\kappa}$-generic, then there exists a topological space $(Y,\tau)$ in $V[G]$ such that $(Y,\tau)\rightarrow (\omega+1)^{1}_{\omega}$, $(Y,\tau)\nrightarrow (\omega^{2}+1)^{1}_{\omega}$ and $\omega_{1} < \chi(Y)=\mathfrak{d}$
\end{cor}

\begin{proof}
	By Lemma \ref{clubfromdiam}, there exists $\alpha \in \omega_1$ such that $A^{m}_{\alpha} \subseteq acc(\omega_{1})$, which remains true in $V[G]$. By Corollary \ref{Cor2} we have that $\vec{A}$ is a $\clubsuit_{F}$-sequence in $V[G]$. If $Y$ is the space obtained by applying Theorem \ref{exampleclub} in $V[G]$ to $\vec{A}$, it follows from Lemma \ref{character} that $\chi(Y)=\mathfrak{d}$. By Corollary \ref{omega1} we have $\mathfrak{d} > \omega_1^{V[G]}$. Therefore $(Y,\tau)$ is the space we wanted. 
\end{proof}

This is interesting because it highlights that even though the spaces obtained in Theorem \ref{Cor1} and Corollary \ref{Yd} have the same underlying set, and are generated considering the same $\clubsuit_{F}$-sequence, they are different. Indeed, the one from our ground model ends up with a coarser topology than the one generated directly in our extension, to the point where they end up having different characters.

\section*{Acknowledgements}

The first author was supported by CAPES (grant agreement 88882.461730/2019-01). The second author was supported by the European Research Council (grant agreement ERC-2018-StG 802756). We are grateful to Rubens Onishi for presenting a series of seminars on the topic of partitions of topological spaces, that brought to our attention the problem addressed at Section \ref{CB}. We would like to thank William Weiss for his suggestion on how to fix the problem mentioned above, as well as introducing us to the line of work developed at Section \ref{Top}. We would also like to thank Assaf Rinot for his comments and suggestions on preliminary versions of this paper.

\end{document}